\documentclass[10pt
]{amsart}
\usepackage{amsmath, amsthm, amssymb,mathrsfs,amsfonts,}
\usepackage{mathtools}
\usepackage{relsize}
\usepackage{enumitem}
\usepackage[utf8]{inputenc}
\usepackage{pgffor}
\usepackage{color}
\usepackage{epstopdf}
\usepackage{graphicx}

\usepackage{varwidth}

\usepackage{verbatim}

\usepackage[
]
{hyperref}
\hypersetup{
	colorlinks=false,       
	linkcolor=blue,          
	citecolor=red,        
	filecolor=magenta,      
	urlcolor=cyan           
}

\newtheorem{thm}{Theorem}[section]
\newtheorem{coro}[thm]{Corollary}
\newtheorem{lem}[thm]{Lemma}
\newtheorem{prop}[thm]{Proposition}

\theoremstyle{definition}
\newtheorem{Def}[thm]{Definition}
\theoremstyle{remark}
\newtheorem{rk}[thm]{Remark}

\numberwithin{equation}{section}

\definecolor{ao}{rgb}{0, 0.5, 0}

\newcommand{\supp} {\text{supp\! }}

\newcommand{\fs}{e^{it(-\Delta)^{\alpha/2}}}

\begin{document}

\title[Pointwise convergence]{Pointwise convergence of sequential  Schr\"odinger means}

\author{Chu-Hee Cho, Hyerim Ko, Youngwoo Koh and Sanghyuk Lee}
\address{Department of Mathematical Sciences and RIM, Seoul National University, Seoul 08826, Republic of Korea}
\email{akilus@snu.ac.kr}
\email{kohr@snu.ac.kr}
\email{shklee@snu.ac.kr}

\address{Department of Mathematics Education, Kongju National University, Kongju 32588, Republic of Korea}
\email{ywkoh@kongju.ac.kr}


\begin{abstract}
We study pointwise convergence   of the fractional Schr\"odinger means along sequences $t_n$ which converge to zero. 
Our main result is that bounds on the maximal function $\sup_{n} |e^{it_n(-\Delta)^{\alpha/2}} f| $ can be deduced from 
those on $\sup_{0<t\le 1} |e^{it(-\Delta)^{\alpha/2}} f|$ when $\{t_n\}$ is contained in the Lorentz space $\ell^{r,\infty}.$
Consequently, our results provide seemingly optimal results in higher dimensions, which extend the recent work of Dimou--Seeger, and Li--Wang--Yan to higher dimensions. 
Our approach based on a localization argument also works for other dispersive equations and provides alternative proofs of previous results on sequential convergence.

\end{abstract}
\maketitle

\section{Introduction}
Let $\alpha>0$. We consider the fractional Schr\"odinger operator
\begin{equation}\label{sol_frac}
e^{it(-\Delta)^{\alpha/2}}f(x) = (2\pi)^{-d} \int_{\mathbb{R}^d} e^{i(x\cdot \xi+t|\xi|^\alpha)} \widehat f(\xi) d\xi. 
\end{equation}
A classical problem posed by Carleson \cite{Carl} is to determine the optimal  regularity $s$ for which
\begin{align}\label{conv}
\lim_{t \rightarrow 0} e^{it(-\Delta)^{\alpha/2}}f = f \quad \text{a.e. }  \quad \forall f \in H^s,
\end{align}
where $H^s$ denotes the inhomogeneous Sobolev spaces of order $s$  with its norm $\|f\|_{H^s(\mathbb R^d)} = \| (1+|\cdot|^2)^{s/2} \widehat{f} \,\|_{L^2(\mathbb R^d)}$. 
The case $\alpha=2$ has been extensively studied until recently. When $d=1$, it was shown by the work of Carleson \cite{Carl} and Kenig--Dahlberg \cite{DK} 
 that \eqref{conv} holds true if and only if $s \ge 1/4$. In higher dimensions,  the problem turned out to be more difficult. 
 Progress was made by contributions of numerous authors.  Sj\"olin \cite{S} and Vega \cite{ V} independently obtained  \eqref{conv} for $s>1/2$. In particular,  further improvement 
 on required regularity   was made by 
 Moyua--Vargas--Vega \cite{MVV}, Tao--Vargas \cite{TV} when $d=2$, and convergence for $s>{(2d-1)}/{4d}$ was shown by Lee \cite{L} for $d=2$ and Bourgain \cite{B} in higher dimensions. 
Bourgain \cite{B2} showed that \eqref{conv} holds only if $s\ge {d}/{2(d+1)}$.  The lower bound was shown to be sufficient for \eqref{conv} by Du--Guth--Li \cite{DGL} for $d=2$, and Du--Zhang \cite{DZ} for $d\ge3$ except for the endpoint case $s= {d}/{2(d+1)}$
(also, see \cite{LR} for earlier results and references therein).

In general, \eqref{conv} continues to be true for $\alpha>1$ if $s> {d}/{2(d+1)}$ (see e.g., \cite{L, CK}, also see \cite{MYZ}).  
If $\alpha=1$,  it is easy to show \eqref{conv} holds if and only if $s>1/2$ in all dimensions (\cite{Co,W2}).
When $0<\alpha<1$,  it is known that \eqref{conv} holds if  $s>\alpha/4$ and the convergence fails if $s<\alpha/4$  in $\mathbb R^1$.  However,  only partial results are known in higher dimensions, i.e., \eqref{conv} holds true for $s>\alpha/2$ and fails for $s<\alpha/4$ when $d\ge2$ (see \cite{W1, W2}).

\subsection*{Convergence along sequences}  
Recently, pointwise convergence along sequences was considered by several authors \cite{DS, S2, SS, SS2, SS3}. More precisely, studied was 
the problem to determine the regularity exponent $s$ such that, for a given sequence $\{t_n\}$ satisfying $\lim_{n\rightarrow \infty} t_n=0$,
\begin{equation}\label{modi_question}
\lim_{n \rightarrow \infty} e^{it_n(-\Delta)^{\alpha/2}}f(x) = f(x) \quad \text{a.e. } x, \quad \forall f\in H^s.
\end{equation}

Naturally, one may expect  that the more rapidly the sequence $\{t_n\}$ converges to zero, the less regularity is required to guarantee almost everywhere convergence.
To quantify how fast the sequence converges to zero,  the sequences in $\ell^r(\mathbb N)$  and  $\ell^{r,\infty}(\mathbb N)$ were considered, where 
$\ell^{r,\infty}(\mathbb N)$  denotes the Lorentz space
\[
\ell^{r,\infty}(\mathbb N):=\big\{ t_n \, : \, \sup_{\delta>0} \delta^r \#\{ n \in \mathbb N: |t_n| \geq \delta \} <\infty \big\}
\]
for $r<\infty$. Note that $\{ n^{-b}\} \in \ell^{r,\infty}(\mathbb{N})$ if and only if $b \ge 1/r$. In particular, Dimou and Seeger \cite{DS} studied the almost everywhere convergence problem in $\mathbb R^1$ using $\ell^{r,\infty}(\mathbb N)$. 
They proved that \eqref{modi_question} holds for all $f \in H^s$
if and only if  $s\ge \min\{ \frac{r\alpha}{4r+2}, \frac{1}{4} \}$ (when $\alpha>1$), $s \ge \frac{r\alpha}{4r+2}$ (when $0<\alpha<1$), and $s\geq \frac{r}{2(r+1)}$ (when $\alpha=1$) for a strictly decreasing convex sequence $\{t_n\} \in \ell^{r,\infty}(\mathbb N)$.
There are also results  in higher dimensions by Sj\"olin \cite{S2}, Sj\"olin and Str\"omberg \cite{SS, SS2, SS3}.
Recently, Li, Wang, and Yan \cite{LWY}, relying on the bilinear approach in \cite{L}, obtained some partial results  for the case $d=\alpha=2$  without assuming  that $\{t_n\}$ decreases.

\subsection*{Maximal estimates} In the study of pointwise convergence  the associated maximal functions play important roles. By a standard argument  \eqref{conv} follows 
if we have 
\begin{align}\label{maxx}
\big\| \sup_{0<t \le 1} |\fs f| \big\|_{L^2(B(0,1))}
\leq C \|f\|_{H^s},
\end{align}
where $B(x,r)=\{ y \in \mathbb R^d: |x-y| <r\}$. 
Likewise,  \eqref{modi_question} follows  from the estimate
\begin{align}\label{tn maxx}
\big\| \sup_{t_n } \big| e^{it_n(-\Delta)^{\alpha/2}}f \big| \big\|_{L^2(B(0,1))}
\leq C  \|f\|_{H^s},
\end{align}
which is, in fact, essentially equivalent to \eqref{modi_question} by Stein's maximal principle. 
Our first result shows that  the maximal estimate  \eqref{tn maxx} can be deduced from \eqref{maxx} when $\{t_n\} \in \ell^{r,\infty}(\mathbb N)$.

\begin{thm}\label{thmpt}
Let $d\geq1$, $\alpha>0$,   $s_*>0$ and $0<r<\infty$.
Suppose \eqref{maxx} holds  for $s\ge s_*$. Then, if $\{t_n\} \in \ell^{r,\infty}(\mathbb N)$,  \eqref{tn maxx} holds  provided
\begin{align}\label{rel}
s\ge\min \Big\{ \frac{r\alpha}{r\min \{ \alpha,1\}+2s_*}s_*, \  s_* \Big\}.
\end{align}
\end{thm}

Thanks to Theorem \ref{thmpt} we can improve the previous results and obtain seemingly optimal results for the convergence of sequential Schr\"odinger means in higher dimensions.
For $\{t_n\} \in \ell^{r,\infty}(\mathbb N)$ and $d=1$,  the known estimates \eqref{maxx} (\cite{Carl, W1})  and Theorem \ref{thmpt} give
\eqref{tn maxx} for $s\ge \min ( r\alpha/(4r+2), 1/4)$ when $\alpha>1$, and   for $s>r\alpha /(4r+2)$ when $0<\alpha<1$. 
This recovers the  result (sufficiency part except the endpoint case when $0<\alpha<1$)  in \cite{DS}  without the assumption that $\{t_n\}$ decreases.

In higher dimensions $d\ge 2$, by combining Theorem \ref{thmpt} and recent progress on the maximal bounds, i.e.,  \eqref{maxx} for $\alpha>1$  and $s>\frac{d}{2(d+1)}$ 
\cite{DGL, DZ, CK},  
we have  the estimate \eqref{tn maxx} for 
\begin{equation}
\label{scon}
s> \min \big\{\frac{\alpha dr}{2(d+1)r+2d},\frac{d}{2(d+1)} \big\}
\end{equation}
whenever $\{t_n\} \in \ell^{r,\infty}(\mathbb N)$.  As a consequence,  we have the following  result on pointwise convergence. 

\begin{coro}\label{thmS}
Let $d\geq 2$, $\alpha>1$ and $0<r<\infty$.
For any sequence $\{t_n\}\in \ell^{r,\infty}(\mathbb{N})$,
\eqref{modi_question} holds for all $f\in H^s(\mathbb{R}^d)$ if \eqref{scon} holds. 

\end{coro}

This improves the previous results  in  \cite{LWY, SS2}. 
We expect that  the regularity exponents given in \eqref{scon} is sharp up to the endpoint case. However,  we are not able to verify  it for the moment.

\begin{rk}
As mentioned before, when $0<\alpha<1$ and $d\ge 2$, it is known that \eqref{conv} holds if $s>\alpha/2$ (\cite{W2}). Thus,  
Theorem \ref{thmpt} yields \eqref{modi_question} for $s> \frac{r\alpha}{2(r+1)}$.
The implication in Theorem \ref{thmpt} also works for more general operators (see Remark \ref{g-symbol}).  In particular, 
one can also recover the result of Li, Wang and Yan \cite{LWY2} for the nonelliptic Schr\"odinger operator by combining Theorem \ref{thmpt} with the results in \cite{RVV}.
\end{rk}
\begin{rk}
For the wave operator, i.e., $\alpha=1$, \eqref{modi_question} holds true if and only if $s \ge \frac{r}{2(r+1)}$ for $\{t_n\} \in \ell^{r,\infty}(\mathbb N)$. When $d=1$, this was shown in \cite{DS}. In higher dimensions, one can show it using Theorem \ref{thmpt} (also Corollary \ref{imp-Hs}). The sharpness  can be obtained by following the argument in \cite{DS}.
We remark that 
\eqref{modi_question} is closely related to $L^p$ boundedness of the spherical maximal operator given  by taking the supremum over more general sets (see \cite{SWW, AHRS, RS}).
\end{rk}

\subsection*{Localization argument}\label{section1.2}
The proof of Theorem \ref{thmpt} relies on a localization argument.  We briefly explain our approach. 
Via Littlewood-Paley decomposition,    the proof of \eqref{modi_question} can be  reduced to showing 
\begin{align}\label{tn max}
\big\| \sup_{t_n } \big| e^{it_n(-\Delta)^{\alpha/2}}f \big| \big\|_{L^2(B(0,1))}
\leq C R^{s} \|f\|_{L^2(\mathbb{R}^d)}
\end{align}
where $\widehat f$ is supported in $\mathbb A_R:=\{ \xi: R\le |\xi|\le 2R\}$ (see Section  \ref{section4}). 
In the previous work \cite{DS, S2, SS, SS2}  the estimate \eqref{tn max} was obtained by relying  on the kernel estimates.
In contrast, we deduce \eqref{tn max} directly from \eqref{maxx}. 
Clearly, \eqref{maxx}  gives  
\begin{align}\label{max22'}
\big\| \sup_{0<t \le 1} |\fs f| \big\|_{L^2(B(0,1))}
\leq C R^{s} \|f\|_{L^2(\mathbb{R}^d)}
\end{align}
for $R\ge1$ whenever $\widehat f$ is supported in $\mathbb A_R$. Using the estimate and a localization argument, 
 we first obtain from \eqref{max22'} a  temporally localized maximal estimate  
\begin{align}\label{main2_cor}
\big\| \sup_{t \in I} |\fs f| \big\|_{L^2(B(0,1))}
\leq C' (1+R^\alpha|I|)^{\max\{s, \frac s \alpha \}} \|f\|_{L^2(\mathbb{R}^d)}
\end{align}
for $R\ge1$ and any subinterval $I \subset [0,1]$ with $|I| \leq R^{1-\alpha}$  when $\widehat f$ is supported in $\mathbb A_R$.
Moreover, the converse implication from \eqref{main2_cor} to \eqref{max22'} is also true as long as $R^{-\alpha}<|I| \le R^{1-\alpha}$ (see Lemma \ref{FF} for detail).
Once we have  \eqref{main2_cor}, we can obtain \eqref{tn maxx} by following the argument in \cite{DS}.  

If the exponent $s$ in the estimate \eqref{max22'} is sharp, then the same is true for the estimate \eqref{main2_cor}. 
For instance, when $\alpha=2$, \eqref{max22'} holds for $s>{d}/ 2(d+1)$, which is optimal up to the endpoint case, and hence so does \eqref{main2_cor} for the same $s$.
When $\alpha>1$ and $|I| \geq R^{1-\alpha}$, one can see 
the exponent s in \eqref{main2_cor} can not be smaller than that in \eqref{max22'} using  the  localization lemma in \cite{L} (cf. \cite{CLV, LeeRogers, Rogers}).

To show the implication from \eqref{max22'} to \eqref{main2_cor}, we adapt the idea of temporal localization lemma in \cite{L, CLV}.
We establish a spatial localization lemma (Lemma \ref{loc}), which plays a crucial role in proving Theorem \ref{thmpt}.
More precisely, we show that the local-in-spatial estimate \eqref{max22'} can be extended to the global-in-spatial estimate with the same regularity exponent.
After a suitable scaling, we  obtain the temporal localized estimate \eqref{main2_cor} from the global-in-spatial estimate.

\subsection*{Extension to fractal measure}  Maximal estimates relative to general measures (instead of the Lebesgue measure) have been used to get more precise description on the pointwise behavior of the Schr\"odinger mean $e^{it(-\Delta)^{\alpha/2}}f$. 
For a given sequence $\{t_n\}$ converging to zero, we consider 
\[
D^{\alpha,d}(f,\{t_n\}) =\big\{ x\in \mathbb R^d:  e^{it_n(-\Delta)^{\alpha/2}}f (x) \not \to f(x)
\quad~ \text{as} \quad~ t_n \rightarrow 0 \big\}
\]
and set
\[
\mathfrak D^{\alpha,d}(s,r)=\sup_{f \in H^s,~ \{t_n\} \in \ell^{r,\infty}} \dim_H D^{\alpha,d}(f,\{t_n\}),
\]
where $\dim_H$ denotes the Hausdorff dimension.
One can compare $\mathfrak D^{\alpha,d}(s,r)$ with the dimension of the divergence set
\[\mathfrak D^{\alpha,d}(s)=\sup_{f \in H^s} \dim_H \big\{ x\in \mathbb R^d:  e^{it_n(-\Delta)^{\alpha/2}}f (x) \not \to f(x)
\quad~ \text{as} \quad~ t \rightarrow 0 \big\}.\]
The bounds on  $\mathfrak D^{\alpha,d}(s)$ can be obtained by the maximal estimate relative to general measures (see, for example,  \cite{BBCR, DZ, HKL}), to which the fractal Strichartz estimates studied in \cite{ CHL, Harris, DGOWWZ} are closely related (also see \cite{Mattila, Wolff, Er}).  

The  implication in  Theorem \ref{thmpt} also extends to the maximal estimates relative to general fractal  measures, so we can  make use of the known estimates for the $L^2$-fractal maximal estimates and the fractal Strichartz estimates to obtain upper bounds on $\mathfrak D^{\alpha,d}(s,r)$, $0<r<\infty$.
We discuss it in detail in Section \ref{sec:div}.

\subsubsection*{Organization of the paper}
In Section \ref{section2}, we deduce from \eqref{maxx} temporally localized maximal estimates in Lemma \ref{FF} (relative to general measure) which are to be used to prove Theorem \ref{thmpt}.
We prove Theorem \ref{thmpt}  and discuss upper bounds on the dimension of divergence sets in Section \ref{section4}.

\subsubsection*{Notations}
Throughout this paper,
a generic constant $C>0$ depends only on dimension $d$, which may change
from line to line.
If a constant depends on some other values (e.g. $\epsilon$), we denote it by $C_\epsilon$.
The notation $A \lesssim B$ denotes $A \leq CB$ for a constant $C>0$, and we denote by $A \sim B$ if $A \lesssim B$ and $B \lesssim A$.
We often denote $L^2(\mathbb{R}^d)$ by $L^2$, and similarly ${H}^s(\mathbb{R}^d)$ by ${H}^s$.

\section{Temporally  localized maximal estimates}\label{section2}

In this section, we prove that the estimates \eqref{max22'} and \eqref{main2_cor} are equivalent.  For later use, we consider 
the equivalence in a more general setting, that is to say,  in the form of estimates  relative to fractal  measures (Lemma \ref{FF}). To do this, we recall the following.

\begin{Def}
Let $0<\gamma \le d$ and  let $\mu$ be a nonnegative Borel measure. We  say $\mu$ is $\gamma$-dimensional if 
there is a constant $C_\mu$ such that
\begin{align}\label{mu}
\mu (B(x,r)) \le C_\mu r^\gamma, \quad \forall (x,r) \in \mathbb R^d \times \mathbb R_+.
\end{align}
By $\langle\mu \rangle_\gamma$ we denote the infimum of such a constant $C_\mu$. 
\end{Def}

\newcommand{\nm}[1]{\langle #1\rangle_\gamma^\frac12}

We first deduce a temporally localized maximal estimate  from the  estimate  
\begin{align}\label{max22}
\big\| \sup_{0<t \le 1} |\fs f| \big\|_{L^{2}(B(0,1),d\mu)} \le C
R^{\frac{d-\gamma}2+s} \nm\mu \|f\|_{L^2}
\end{align}
which holds whenever $\widehat f $ is supported on $\mathbb A_R$. 

\begin{lem}\label{FF}
Let $p\ge2$, $R\ge1$, $\alpha>0$, $0<\gamma \le d$.  Let $I \subset [0,1]$ be an interval such that $ |I| \le \min \{ R^{1-\alpha},1\}$. 
Suppose that \eqref{max22} holds for some $s$ whenever $\widehat f$ is supported in $\mathbb A_{R}$ and $\mu$ is a $\gamma$-dimensional measure in $\mathbb R^d$. 
Then, for any $\gamma$-dimensional measure $\widetilde \mu$ in $\mathbb R^d$,
 there is a constant $C'>0$ such that
\begin{align}\label{main2}
\big\| \sup_{t \in I} |\fs f| \big\|_{L^{2}(B(0,1),d\widetilde\mu)}
\le
C' \nm{\widetilde\mu}R^{\frac{d-\gamma}2} (1+R^\alpha|I|)^{\max\{s, \frac s \alpha \}} \|f\|_{L^2}
\end{align}
holds whenever $\widehat f$ is supported on $\mathbb A_R$.
Conversely, if \eqref{main2} holds for $\widehat f $ supported on $\mathbb A_R$, $\widetilde\mu$ is  $\gamma$-dimensional, and interval $I\subset [0,1]$ satisfies  $R^{-\alpha} < |I| \le \min \{ R^{1-\alpha},1\}$,  then there exists $C>0$ such that \eqref{max22} holds whenever $\widehat f$ is supported on $\mathbb A_R$ and $\mu$ is $\gamma$-dimensional.
\end{lem}

\begin{rk}\label{g-symbol}
By a simple modification of our argument, Lemma \ref{FF} can be extended to  a class of evolution operators  $e^{itP(D)}$ as long as 
\[  |\partial_\xi^\beta P(\xi)|\lesssim  |\xi|^{\alpha-|\beta|}, \quad \forall \beta  \] 
and $|\nabla   P(\xi)|\gtrsim |\xi|^{\alpha-1}$ hold (see \cite{CLV}).
Hence, an analogue of Theorem \ref{thmpt} holds true for $e^{itP(D)}$. A typical example of such an operator is the non elliptic Schr\"odinger operator $e^{it(\partial_{x_1}^2-\partial_{x_2}^2\pm \partial_{x_3}^2\pm \dots \pm \partial_{x_d}^2)}$. 
\end{rk}

The rest of the section is devoted to the proof of  Lemma \ref{FF}, for which we first consider  spatial localization.

\subsection{Spatial localization}

By adapting the argument in \cite{L, CLV}, we prove a spatial localization lemma exploiting rapid decay of the kernel.

\begin{lem}\label{loc}
Let $\alpha>0$ and $r\ge1$. Let $\mu$ be a $\gamma$-dimensional measure in $\mathbb R^d$.
For $R \ge1$,  we set $I_R=[0,R]$. 
Suppose that
\begin{equation}\label{small ball}
\big\| \fs f \big\|_{L_x^{2}(B(0,R), d \mu; L_t^r(I_R))}
	\le C  R^s \nm\mu \|f\|_{L^2}
\end{equation}
holds for some $s \in \mathbb R$ whenever $\widehat f$ is supported in $\mathbb A_{1}$.
Then, there exists a constant $C_1>0$ such that
\begin{equation}\label{large ball}
\big\| \fs f \big\|_{L_x^{2}(\mathbb R^d,d \mu; L_t^r(I_R))}
	\le C_1 R^s \nm\mu \|f\|_{L^2}
\end{equation}
holds whenever $\widehat f$ is supported in $\mathbb A_{1}$. 
\end{lem}

\begin{proof}
Let $P$ be a projection operator defined by $\mathcal F(Pg)(\xi)=\beta(|\xi|)\widehat g(\xi)$ where $\beta\in C_c(2^{-1},2^2)$ and $\beta=1$ on $[1,2]$.
Let $\{B\}$ be a collection of finitely overlapping balls of radius $R$ which cover $\mathbb R^d$. Denote $\widetilde B=B(a,10\alpha 2^{2|\alpha-1|}R)$ if $B=B(a,R)$. 
Then, we note that $\|F\|_{L^{2}(\mathbb R^d)}^2 \lesssim \sum_B \|F\|_{L^{2}(B)}^2$.

Since  $Pf=f$, by Minkowski's inequality we have
\[\big\| \fs f \big\|_{L_x^{2}(\mathbb R^d, d \mu; L_t^r(I_R))}^2
\lesssim (\mathcal L_1+\mathcal L_2),\]
 where
\begin{align*}
\mathcal L_1&= \sum_B \big\| \fs P(\chi_{\widetilde B}f)\big\|_{L_x^{2}(B,d \mu; L_t^r(I_R))}^2, \\
\mathcal L_2&= \sum_B \big\| \fs P(\chi_{\widetilde {B}^{\mathrm c}}f) \big\|_{L_x^{2}(B,d\mu; L_t^r(I_R))}^2.
\end{align*}

Note that $\fs Pf=K(\cdot,t)\ast f$ where $K$ is given by
\begin{equation*}
K(x, t) = \int e^{i (x \cdot \xi + t |\xi|^\alpha) } \beta(|\xi|)\,d\xi.
\end{equation*}
By integration by parts, it is easy to see $|K (x,t)| \le C_NR^{-N} (1+|x|)^{-N}$ for any $N\ge1$
if $|x|>10\alpha 2^{2|\alpha-1|} R$ and $|t| \le R$.
Thus,
we have
\[ \| \fs (\chi_{\widetilde {B}^{\mathrm c}}f)(x)\|_{L_t^r(I_R)} \le  C_N R^{-N} (1+|\cdot|)^{-N} \ast |f|(x)\] 
for any $N\ge 0$ whenever $x \in B$.  Taking $N$ large enough, we get  
\[ \mathcal L_2\le C  R^{2s} \| (1+|\cdot|)^{-(d+\gamma)} \ast |f| \|_{L_x^{2}(\mathbb R^d,d\mu)}^2 . \]
By  Schur's test it follows that $ \| (1+|\cdot|)^{-(d+\gamma)} \ast |f| \|_{L_x^{2}(B,d\mu)}\le C\langle \mu \rangle^\frac12\|f\|_2$. 
Therefore,  we only need to consider  $\mathcal L_1$. 
 
 Applying  \eqref{small ball} on each $B$, we obtain
\[
\textstyle \mathcal L_1 \lesssim R^{2s} \langle \mu \rangle \sum_B \|\chi_{\widetilde B} f\|_{L^2}^2 \le C  R^{2s} \langle \mu \rangle \|f\|_{L^2}^2.
\]
The last inequality  follows since  the balls $\widetilde B$ overlap finitely. This completes the proof. 
\end{proof}

\subsection{Proof of Lemma \ref{FF}}  To prove Lemma \ref{FF}, we invoke an elementary lemma.

\begin{lem}[{\cite{HKL}}]\label{easy}
Let $\mu$ be a $\gamma$-dimensional measure in $\mathbb R^d$.
If $\widehat F$ is supported on $B(0,R)$, then
$
\| F\|_{L^{2}(d\mu)}\le C 
R^{\frac{d-\gamma}2} \nm\mu \|F\|_{L^{2}(\mathbb R^d)}.
$
\end{lem}

By  translation and Plancherel's theorem, we may assume that $I=[0,\delta]$ with $\delta \le \min\{ R^{1-\alpha},1\}$. 
We  may further assume $R^{-\alpha}<\delta$ since \eqref{main2} follows by the Sobolev embedding and Lemma \ref{easy} if  $\delta \le R^{-\alpha}$.

For a givne $\gamma$-dimensional measure $\mu$, 
we denote by $\mu_R$ the  measure defined by the relation
\footnote{$\mu_R$ is a positive Borel measure by the Riesz representation theorem.}
\begin{align}\label{mu-r}
\int F(x)\,d\mu_R(x)= R^{\gamma} \int F(R x)\,d\mu(x), \quad F \in C_0(\mathbb R^d).
\end{align}
It is easy to see that $\mu_R$ is a $\gamma$-dimensional measure in $\mathbb R^d$
such that 
\[ \mu_R(B(x,r))\le C \langle\mu \rangle_\gamma r^\alpha\]  for some $C>0$.
Changing variables $(x,t) \rightarrow (R^{-1}x, R^{-\alpha} t)$ and $\xi \rightarrow R\xi$, we see that \eqref{max22} is equivalent to 
\begin{align}\label{dil}
\textstyle \big\| \sup_{t \in[0, R^\alpha]} | e^{it(-\Delta)^{\alpha/2}} f_R | \big\|_{L^{2}(B(0,R),d \mu_R)} \le C
R^s \nm\mu \|f_R \|_{L^2}
\end{align}
where $\widehat {f_R}(\xi)=R^{\frac d2} \widehat f(R \xi)$.
Note that $\|f_R\|_2=\|f\|_2$ and $\widehat {f}_R$ is supported on $\mathbb A_1$.
Let us denote $R' =\min \{R, R^\alpha\}$.  We claim that  the estimate \eqref{dil} is equivalent to the seemingly weaker estimate 
\begin{align}\label{dil2}
\textstyle \big\| \sup_{t \in[0, R']} | e^{it(-\Delta)^{\alpha/2}} g | \big\|_{L^{2}(B(0,R'),d \mu_R)} \le C
R^s\nm\mu \|g \|_{L^2}
\end{align}
for $R\ge1$ whenever $\widehat g$ is supported on $\mathbb A_1$. 
To show this, we only need to prove  that  \eqref{dil2} implies \eqref{dil} since the converse  is trivially true.
When $\alpha>1$, the implication from \eqref{dil2} to \eqref{dil} was shown in  \cite{CLV} (also, see \cite{L, MYZ}) when $\mu_R$ is the Lebesgue measure and $\alpha$ is an integer.
It is easy to see that the argument in \cite{CLV} works for general $\gamma$-dimensional measure $\mu_R$.
When $0<\alpha\le 1$,  using Lemma \ref{loc} with $R$ replaced by  $R^\alpha$, we get \eqref{dil} from \eqref{dil2}.
This proves the claim.

 We  now show that \eqref{max22} and  \eqref{main2} are equivalent. Recall that we are assuming that $R^{-\alpha}<\delta$. 
Changing variables $(x,t) \rightarrow (R^{-1}x, R^{-\alpha} t)$, $\xi \rightarrow R\xi$  in \eqref{main2} as above, we see that
\eqref{main2} is equivalent to 
\begin{align}\label{max-2}
\big\| \sup_{t \in[0, R^\alpha\delta]} | e^{it(-\Delta)^{\alpha/2}} f_R| \big\|_{L^{2}(B(0,R),d \widetilde\mu_R)} \le C'
(R^\alpha \delta)^{\max\{ s, \frac s \alpha \}} \nm {\widetilde\mu} \|f_R\|_{L^2}.
\end{align}
By Lemma \ref{loc} with $R$ replaced by $R^\alpha \delta$,  \eqref{max-2} follows from 
\begin{align}\label{dil3}
\big\| \sup_{t \in[0, R^\alpha\delta]} | e^{it(-\Delta)^{\alpha/2}} f_R| \big\|_{L^{2}(B(0,R^{\alpha}\delta),d \widetilde\mu_R)} \le C'
(R^\alpha \delta)^{\max\{ s, \frac s \alpha \}} \nm{\widetilde \mu} \|f_R\|_{L^2}.
\end{align}
Thus, \eqref{dil3} and \eqref{max-2} are trivially  equivalent.  Therefore,  to show the equivalence of \eqref{max22} and  \eqref{main2}, 
it is enough to prove that of \eqref{dil2} and \eqref{dil3}. Indeed, 
it is clear that  \eqref{dil3} follows from  \eqref{dil2}   by replacing $R'$ in \eqref{dil2} with $R^\alpha \delta$.
Conversely, if we replace  $R^\alpha\delta$ in \eqref{dil3} with $R'$,
 we get  \eqref{dil2} as long as $\delta>R^{-\alpha}$.
\qed

\section{Maximal estimate for sequential Schr\"odinger means}\label{section4}

In this section, we prove Theorem \ref{thmpt} and obtain results regarding upper bounds on the dimension of the divergence set of $e^{it_n(-\Delta)^{\alpha/2}}f $. 
The results are consequence of extension of  the maximal estimates  to general measure. 
See Section \ref{sec:div}.

\subsection{$L^2$-maximal estimates}\label{subsec31}
Making use of Lemma \ref{FF}, we deduce the maximal estimates for the sequential Schr\"odinger mean from the estimate  \eqref{max22}.

\begin{prop}\label{sch}
Let $R\ge1$, $\alpha>0$, and $0<r<\infty$. Suppose  that \eqref{max22} holds for some $s=s_*>0$ whenever $\mu$ is a $\gamma$-dimensional measure in $\mathbb R^d$
and   $\supp \widehat f \subset \mathbb A_R$. Let 
\begin{align}\label{ssr}
\tilde s_*=\frac{r\alpha}{r\min\{\alpha,1\}+2s_*}s_*.
\end{align}
Then, if
$\{t_n\} \in \ell^{r,\infty}$,  there is a constant $C>0$ such that
\begin{align}\label{L2-L2}
\|\sup_n |e^{it_n(-\Delta)^{\alpha/2}}f| \|_{L^{2}(B(0,1),d\mu)}
\le C R^{\frac{d-\gamma}2+s}  \nm{\mu} \|f\|_{L^2}
\end{align}
holds for $s\ge\min\{ s_*, \tilde s_{*} \}$ whenever  $\mu$ is  $\gamma$-dimensional and $\supp \widehat f \subset \mathbb A_R$. 
\end{prop}

When $\alpha>1$, Proposition \ref{sch} is meaningful only for $r<2s_\ast/(\alpha-1)$.

\begin{proof}[Proof of Proposition \ref{sch}]
We may assume $\tilde s_* \le s_*$ since \eqref{L2-L2} trivially holds for $s \ge s_*$ by the maximal estimate  \eqref{max22}.
For $0<\delta <1$, let us set 
\[ \mathrm N(\delta)=\{ n\in\mathbb{N} :~ t_n < \delta \}.\]
Since $\{t_n\} \in \ell^{r, \infty}(\mathbb{N})$,
there is a uniform constant $C_0>0$ such that
    \begin{equation}\label{number}
    |\mathrm N(\delta)^{\mathrm c}|
    \leq C_0 \delta^{-r}.
    \end{equation}
Then, it follows that
\[ \big\|\sup_n |e^{it_n(-\Delta)^{\alpha/2}}f| \big\|_{L^2(B(0,1),\mu)} \le \mathcal I+\mathcal{I\!\!I},\]
where
    \begin{align*}
    \mathcal I&= \big\| \sup_{n\in \mathrm N(\delta)} | e^{it_n (-\Delta)^{\alpha/2}}  f| \big\|_{L^2(B(0,1),d\mu)} ,\\
    \mathcal {I\!\!I}&= \big\| \sup_{n\in \mathrm N(\delta)^{\mathrm c}} | e^{it_n (-\Delta)^{\alpha/2}} f| \big\|_{L^2(B(0,1),d\mu)}.
    \end{align*}

We consider $\mathcal{I}$ first. Since $ \sup_{n\in \mathrm N(\delta)} | e^{it_n (-\Delta)^{\alpha/2}}  f|\le  \sup_{0<t\le \delta} | e^{it (-\Delta)^{\alpha/2}}  f|$,  by Lemma \ref{FF} we have
\begin{align}\label{II}
\mathcal {I} \le C
R^{\frac{d-\gamma}2} (R^\alpha \delta)^{\max\{s_*, \frac {s_*} \alpha \}} \nm\mu \|f\|_{L^2}
\end{align}
provided that $\supp \widehat f \subset \mathbb A_R$ and $R^{-\alpha} \le \delta \le \min\{ R^{-\alpha+1},1\}$.
To handle $\mathcal {I\!\!I}$, we first note that 
$$
\|e^{it_n(-\Delta)^{\alpha/2}} f \|_{L^2(B(0,1),d\mu)}\lesssim R^{\frac{d-\gamma}2}  \nm\mu \|f\|_2, 
$$
which  follows by Lemma \ref{easy} and Plancherel's theorem. So, by the embedding $\ell^2 \subset \ell^\infty$, combining the above estimate and \eqref{number}, we obtain
\[
\textstyle \mathcal {I\!\!I} 
\lesssim  \big( \sum_{n\in \mathrm N(\delta)^{\mathrm c}} \big\| e^{it_n (-\Delta)^{\alpha/2}} f \big\|_{L^2(B(0,1),d\mu)}\big)^{1/2}\lesssim  C_0^{\frac 12}\delta^{-\frac r2} R^{\frac{d-\gamma}2} \nm\mu \|f\|_{L^2}.
\]

Now we prove \eqref{L2-L2} by optimizing the estimates  with a suitable choice of $\delta$.  
When $\alpha \ge 1$, we take $\delta=R^{- 2 s_*\alpha /(r+2s_*)}$, which gives \eqref{L2-L2}  for
$
s\ge rs_*\alpha/(r+2s_*) .
$
When $0<\alpha <1$, we choose $\delta=R^{- 2 s_*\alpha /(r\alpha +2s_*)}$ and obtain \eqref{L2-L2} for
$
s\ge rs_*\alpha/(r\alpha+2s_*).
$
In both cases, one can easily check  $R^{-\alpha} \le \delta \le \min \{R^{-\alpha+1},1\}$ for  $r$ and $s_*$ satisfying
$\tilde s_* \le s_*$. Indeed, if $\alpha>1$, then $\delta \le R^{-\alpha+1}$ since $2s_*+r\ge r\alpha $.
When $0<\alpha\le1$, we have $\delta \le 1$ since $s_*>0$.
\end{proof}

Theorem \ref{thmpt} is an immediate consequence of the following. 

\begin{coro}\label{imp-Hs} Let  $0<\gamma \le d$ and $0<r<\infty$. Suppose 
\begin{align}\label{strongL2}
\big\| \sup_{0<t<1} |e^{it(-\Delta)^{\alpha/2}}f| \big\|_{L^{2}(B(0,1),d\mu)}
\le C\nm\mu \|f\|_{H^{\frac{d-\gamma}2+s_*}}
\end{align}
holds  for some  $0<s_*$ whenever $\mu$ is a $\gamma$-dimensional measure in $\mathbb R^d$. Then, if   $\{t_n\} \in \ell^{r,\infty}$,  there is a constant $C'>0$ such that
\begin{align}\label{H-L2}
\|\sup_{n\in\mathbb{N}} |e^{it_n(-\Delta)^{\alpha/2}}f| \|_{L^{2}(B(0,1),d\mu)}
\le C' \nm\mu \|f\|_{H^{\frac{d-\gamma}2+s}}
\end{align}
holds for $s\ge\min\{ s_*, \tilde s_*\}$ where $\tilde s_*$ is given by \eqref{ssr}.
\end{coro}

The estimate \eqref{strongL2} clearly implies \eqref{max22}. However, to prove Corollary \ref{imp-Hs}, we need to remove the frequency localization  in the estimate \eqref{L2-L2} so that 
 the right hand side of \eqref{L2-L2} is replaced by $C\nm\mu \|f\|_{H^{\frac{d-\gamma}2+s}}$.  
   This can be achieved by adapting  the argument in \cite{DS}.

\begin{proof}[Proof of Corollary \ref{imp-Hs}]
As before, we may assume $\tilde s_* \le s_*$. It suffices to show \eqref{H-L2} for $s\ge \tilde s_*$.
Let us choose a smooth function $\beta \in C_0^\infty((1/2,2))$ such that $\sum_k \beta(2^{-k}\cdot)=1$ and set $\beta_0=\sum_{k\le0} \beta(2^{-k}\cdot)$.
Let $P_k$, $k\ge 0$,  be the  projection operator defined by
$
\widehat{P_kf}(\xi)=\beta(2^{-k}|\xi|)\widehat f(\xi)$, $k\ge1$, and $\widehat{P_0f}(\xi)=\beta_0(|\xi|)\widehat f(\xi)$.

For $\ell \ge0$, we set
\[
\mathrm N_\ell=\big\{ n \in \mathbb N: 2^{-2(\ell+1)\tilde s_*/r}< t_n \le 2^{-2\ell \tilde s_*/r} \big\}.
\]
For each $\ell \ge0$, we write
$f=\sum_{0 \le k < \ell} P_kf + \sum_{k\ge 0} P_{\ell+k}f$.  So, we  have
\[\Big\| \sup_{n\in\mathbb{N}} | e^{it_n(-\Delta)^{\alpha/2}}f| \Big\|_{L^2(B(0,1),d\mu)}
\le \mathrm {I} + \mathrm {I\!I},
\]
where
\begin{align*}
\mathrm {I} &= \sup_{\ell\ge0}\Big\| \sup_{n\in\mathrm N_\ell } \big|  e^{it_n(-\Delta)^{\alpha/2}} (\sum_{0 \le k< \ell}P_{k} f) \big| \Big\|_{L^2(B(0,1), d\mu)},\\
\mathrm {I\!I} &= \sup_{\ell \ge0}\Big\| \sup_{n\in\mathrm N_\ell} \big| \sum_{k \ge 0} e^{it_n(-\Delta)^{\alpha/2}}P_{\ell+k} f \big| \Big\|_{L^2(B(0,1),d\mu)}.
\end{align*}

We consider $\mathrm{I\!I}$ first. Since $\{t_n\} \in \ell^{r,\infty}$, it follows that
$|\mathrm N_\ell| \lesssim 2^{2 \tilde s_* \ell}$. As before, by the embedding  $\ell^2 \subset \ell^\infty$ and then applying Lemma \ref{easy} and Plancherel's theorem, we get
\[
\Big\| \sup_{n\in\mathrm N_\ell} \big| \sum_{k \ge 0} e^{it_n(-\Delta)^{\alpha/2}}P_{\ell+k} f \big| \Big\|_{L^2(B(0,1),d\mu)}
\lesssim 
\nm\mu \sum_{k\ge0} \Big( \sum_{\ell\ge0} 2^{(d-\gamma)(\ell+k)+2\tilde s_*\ell}\| P_{\ell+k}f\|_{L^2}^2\Big)^\frac12. 
\]
Thus,  we obtain
\begin{align}\label{sum2}
\mathrm {I\!I}
\lesssim  \nm\mu \sum_{k\ge0} 2^{-\tilde s_* k}
 \| f\|_{H^{\frac{d-\gamma}2+ \tilde s_*}}
\lesssim  \nm\mu \| f\|_{H^{\frac{d-\gamma}2+\tilde s_*}}.
\end{align}

We now turn to $\mathrm{I}$. Note that ${2\tilde s_*}<{r\alpha}$ and  decompose
$$
\sum_{0 \le k< \ell}P_{k} f
=\sum_{0 \le k <\frac{2\tilde s_*}{r\alpha}\ell}P_kf+ \sum_{0 < k \leq \frac{r\alpha-2\tilde s_*}{r\alpha} \ell } P_{\ell-k}f.
$$ 
By  the Minkowski inequality, we have $\mathrm I \le \mathrm I_a+\mathrm I_b$ where
\begin{align*}
\mathrm I_a&= \sum_{k\ge0} \big\| \sup_{\ell>\frac{r \alpha}{2\tilde s_*}k} \, \sup_{n \in \mathrm N_\ell} |e^{it_n(-\Delta)^{\alpha/2}} P_k f| \big\|_{L^2(B(0,1),d\mu)},\\[1ex]
\mathrm I_b&=\sup_{\ell\ge0}\Big\| \sup_{n\in\mathrm N_\ell } \big|  e^{it_n(-\Delta)^{\alpha/2}} (\sum_{0 < k \leq \frac{r\alpha-2\tilde s_*}{r\alpha} \ell}P_{\ell-k} f) \big| \Big\|_{L^2(B(0,1), d\mu)}.
\end{align*}
Regarding $\mathrm I_a$, note that  $\bigcup_{\ell>\frac{r \alpha}{2\tilde s_*}k} \mathrm N_\ell \subset [0, 2^{-\alpha k}]$.
By \eqref{strongL2} and Lemma \ref{FF} with $R=2^{k}$ and $I=[0, 2^{-\alpha k}]$, we obtain 
\begin{align}\label{sum1a}
\mathrm I_a \lesssim \nm\mu  \sum_{ k\ge0} 2^{\frac{d-\gamma}2k} \|P_k f\|_2 \lesssim  \nm\mu \|f\|_{H^{\frac{d-\gamma}2+s}}
\end{align}
for $s>0.$ 

To handle $\mathrm I_b$, we note that $t_n \in J_\ell:=[0, 2^{-2 \ell \tilde s_*/r}]$
if $n \in \mathrm N_\ell$, and $2^{(\ell-k)\alpha} |J_\ell| \ge1$ since $k \le \frac{r\alpha-2\tilde s_*}{r\alpha}\ell$.
By  Lemma \ref{FF} with $R=2^{\ell-k}$ and $I=J_\ell$, it follows that
\[
\mathrm I_b \lesssim \nm\mu  \sup_{\ell \ge0}
\sum_{0 \le k \le \frac{r\alpha-2\tilde s_*}{r\alpha}\ell}  \ (2^{\ell-k})^{\frac {d-\gamma}2} \big(2^{(\ell-k)\alpha-2\ell \tilde s_*/r} \big)^{\max\{s_*,\frac {s_*}\alpha\}} \|P_{\ell-k}f\|_{L^2}.
\]
Using  \eqref{ssr} and the fact that $\min \{\alpha,1\} \times \max\{s_*, \frac{s_*}\alpha\}=s_*$, one can easily see that
$
\big(2^{(\ell-k)\alpha-2\ell \tilde s_*/r} \big)^{\max\{s_*,\frac {s_*}\alpha\}}=2^{-k\alpha \max \{ s_*, \frac{s_*}\alpha\}} 2^{\ell \tilde s_*}.
$
So, by the embedding $\ell^2 \subset \ell^\infty$ and Minkowski's inequality, we get
\begin{align*}
\mathrm I_b& \lesssim \nm\mu \sum_{k \ge 0} \Big( \sum_{\ell >\frac{r\alpha}{r\alpha-2\tilde s_*} k} 
(2^{\ell-k})^{d-\gamma} 2^{2\ell \tilde s_*} 2^{-2k \max \{\alpha s_*, s_*\}}
\| P_{\ell-k} f\|_{L^2}^2 \Big)^{1/2} 
\\
&\lesssim \nm\mu \sum_{k\ge0} 2^{(\tilde s_*- \max\{ \alpha s_*, s_*\})k} \| f\|_{H^{\frac{d-\gamma}2+\tilde s_*}}.
\end{align*}
Thus, we have 
$
\mathrm I_b
\lesssim \nm\mu \| f\|_{H^{\frac{d-\gamma}2+\tilde s_*}}.
$
Combining this and  the estimates \eqref{sum2} and \eqref{sum1a}, we obtain \eqref{H-L2}.
\end{proof}

\subsection{Dimension of divergence set}\label{sec:div}
Via the implication in Corollary \ref{imp-Hs},  we  can obtain upper bounds on the divergence set $\mathfrak D^{\alpha,d}(s,r)$ making use of the known estimates for 
the maximal Schr\"odinger operator  $f\to \sup_{0<t\le 1} |e^{it(-\Delta)^{\alpha/2}}f |$.
We start with recalling the following lemma (\cite{OO, HKL}).

\begin{lem}\label{imply_max}
Suppose
\begin{align}\label{est11}
\| e^{it(-\Delta)^{\alpha/2}}f\|_{L^2(B(0,1),d \nu)} \le C' \nm\nu
R^{\frac{d-\gamma}2+s} \|f\|_{L^2(\mathbb R^d)},
\end{align}
holds for some $s \in \mathbb R$ whenever $\supp \widehat f \subset \mathbb A_R$ and  $\nu$ is a  $\gamma$-dimensional measure in $\mathbb R^{d+1}$.
Then there exists $C>0$ such that \eqref{max22} holds whenever $\mu$ is a $\gamma$-dimensional measure in $\mathbb R^d$.
\end{lem}

In what follows, we summarize the currently available maximal estimates \eqref{max22} which can be obtained by the best known fractal Strichartz estimates \eqref{est11} combined with Lemma \ref{imply_max}.

\begin{prop} \label{temp max}
Let $d\ge1$, $\alpha\in (0,1)\cup(1,\infty) $,  and $\mu$ be a $\gamma$-dimensional measure in $\mathbb R^d$.
Then \eqref{max22} holds for $s>s_\alpha(\gamma,d)$ where 
\begin{align*}
s_\alpha(\gamma,d)
=
\begin{cases}
\min\big\{ \max\big\{ 0, ~\frac \gamma 2- \frac{d}4 \big\}, ~\frac{\gamma}{2(d+1)} \big\}, &
\quad \text{when }  \alpha>1, 
\\
\min\big\{ \max\{ 0, \alpha(\frac\gamma2-\frac d4) \big\}, \frac \alpha 2\big\}, & \quad \text{when } 0< \alpha< 1.
\end{cases}
\end{align*}
\end{prop}

When $\alpha>1$,  the estimate \eqref{max22} and \eqref{est11} for $s>s_\alpha(\gamma,d)$ were already obtained in \cite{BBCR, Mattila, DZ}).
The estimate for $0<\alpha<1$ also can be shown by the standard argument in \cite{S3,Er} and Lemma \ref{imply_max}. However, since the latter  case is less well known,  we provide the proof for the convenience of the reader.

\begin{proof}[Proof of  Proposition \ref{temp max} for $0<\alpha <1$]
By Lemma \ref{imply_max} we only need to prove \eqref{est11} for $s>s_\alpha(\gamma,d)$.

Let  $\sigma$ be the surface measure on $\{(\xi, |\xi|^\alpha): 1/2\le |\xi| \le 2\}$.
In \cite{S3, Er}, it was shown  that 
\begin{align}\label{dec}
\int |\widehat \nu(R\eta)|^2	\,d\sigma(\eta) \lesssim I_\gamma(\nu) R^{-\beta}
\end{align}
holds for $\beta=\max\{ \min\{\gamma, \frac{d}2\}, \gamma-1\}$ whenever $\nu$ is a $\gamma$-dimensional measure in $\mathbb R^{d+1}$. 
Here $I_\gamma(\nu)$ denote the $\gamma$-dimensional energy of $\nu$.
Let $\nu_\lambda$ be the rescaled measure defined by the relation \eqref{mu-r} with $d$ replaced by $d+1$. 
Then, it is easy to see (see e.g., \cite{HKL, Harris}) that   \eqref{dec} implies the estimate 
\begin{align}\label{fracR}
\| e^{it(-\Delta)^{\alpha/2} }g\|_{L^2(B(0,\lambda),d \nu_\lambda)} \le C
\lambda^{s}\nm\nu \|g\|_2
\end{align}
for  $s>({\gamma-\beta})/2$ whenever $\nu$ is $\gamma$-dimensional and $\widehat g$ is supported on $\mathbb A_1$ (\cite{Wolff, Er2}).
Therefore, we have  \eqref{fracR}   for $s>s_\alpha(\gamma,d)/\alpha$.

Now we take $\lambda=R^\alpha$ in \eqref{fracR}. Then, 
applying Lemma \ref{loc} with $R$ replaced by $R^\alpha$, we have
\begin{align*}
\| e^{it(-\Delta)^{\alpha/2} }g\|_{L^2(B(0,R)\times [0,R^\alpha],d\nu_{R^\alpha})} \le C 
R^{s} \nm\nu\|g\|_2
\end{align*}
for $s>s_\alpha(\gamma,d)$.
By rescaling $\xi \rightarrow R^{-1}\xi$ and $(x,t)\rightarrow  (Rx, R^\alpha t)$, we see that \eqref{est11} holds
for $s>s_\alpha(\gamma,d)$.
\end{proof}

We recall the maximal estimate \eqref{max22} for the wave operator shown in \cite{BBCR, HKL}. (See also \cite{CLV, Harris} for the fractal Strichartz estimates \eqref{est11}).
\begin{prop}\label{temp max2}
Let $I$ be a subinterval in $[0,1]$.
Then \eqref{max22} holds with $\alpha=1$ for $s>s_1(\gamma,d)$ where
\begin{align*}
s_1(\gamma,d)&=
\begin{cases}
\ \ 0, \ \ \ \ \ & \ \quad 0<\gamma \le \frac{1}2,\\[.5ex]
\ \frac \gamma2-\frac 14,\ \ \ \ \ &  \,  \quad \frac12<\gamma \le 1, \\[.5ex]
\ \  \frac \gamma 4, \ \ \ \ \ &\ \quad 1<\gamma \le 2,\\
\end{cases}
\ \ \quad \text{ for }  d=2; 
\\[.5ex]
s_1(\gamma,d)&=
\begin{cases}
\ \ 0, & \ \ \ \quad 0<\gamma \le \frac{d-1}2,\\[.5ex]
\frac {-d+1+2\gamma}{8}, &\quad \frac{d-1}2<\gamma \le \frac{d+1}2,\\[.5ex]
\ \frac{\gamma-1}{2(d-1)}, &\quad \frac {d+1}2 < \gamma \le d, 
\end{cases}
\quad \text{ for }  d\ge 3.
\end{align*}
\end{prop}

By Proposition \ref{sch}, the estimates in Proposition \ref{temp max} and \ref{temp max2}  give the corresponding estimates for  $\sup_n |e^{it_n(-\Delta)^{\alpha/2}}f|$ with $\{t_n\} \in \ell^{r,\infty}$ 
relative to $\gamma$-dimensional measures.  Then, by a standard argument (see \cite{BBCR}), one can obtain  upper bounds on the Hausdorff dimension of the divergence sets.
We summarize the results as follows: 

\begin{coro}
Let $\alpha>0$, $d\geq1$, $r\in(0,\infty)$, and $0<\gamma\leq d$.
Let $s_\ast=s_\alpha(\gamma,d)$ which is given in Proposition \ref{temp max} and \ref{temp max2}.
Then,
$\mathfrak D^{\alpha,d}(s,r) \leq \gamma$  if
$s> 2^{-1}(d-\gamma) +\min\{s_\ast, \tilde s_\ast\}$.
\end{coro}

\subsection*{Acknowledgement}
This work was supported by the NRF (Republic of Korea) grants  No. 2020R1I1A1A01072942 (Cho), No. 2022R1I1A1A01055527 (Ko), No. 2022R1F1A1061968(Koh), and No. 2022R1A4A1018904 (Lee).


\end{document}